\documentclass[a4paper,12pt]{amsart}

\usepackage{latexsym}
\usepackage{amssymb}
\usepackage{graphicx}
\usepackage{amsthm}
\usepackage{epsfig}
\RequirePackage{color}

\usepackage{amscd,enumerate}
\usepackage{amsfonts}

\input xy
\xyoption{all}

\usepackage{multicol}
\usepackage{hyperref}
\hypersetup{ colorlinks,
linkcolor=blue,
filecolor=green,
urlcolor=blue,
citecolor=blue }

\newtheorem{lemma}{Lemma}[section]
\newtheorem{corollary}[lemma]{Corollary}
\newtheorem{theorem}[lemma]{Theorem}
\newtheorem{proposition}[lemma]{Proposition}
\newtheorem{remark}[lemma]{Remark}
\newtheorem{definition}[lemma]{Definition}
\newtheorem{definitions}[lemma]{Definitions}
\newtheorem{example}[lemma]{Example}
\newtheorem{examples}[lemma]{Examples}

\newcommand{\supp}{\operatorname{supp}}
\def\a{\alpha}

\def\N{\mathbb N}
\def\geo{G_E^{(0)}}
\def\uhs{U_{H,S}}
\def\te{\mathcal{T}_E}
\def\oe{\mathcal{O}_E}

\def\lgr{{\mathcal L}_{gr}}

\definecolor{turquoise2}{rgb}{0,0.898039,0.933333}
\definecolor{magenta}{rgb}{1,0,1}

\begin{document}


\title{Using Steinberg algebras to study decomposability of  Leavitt path algebras}

\author[L.O. Clark]{Lisa Orloff Clark}
\address{L.O. Clark: Department of Mathematics and Statistics, University of Otago, PO Box 56, Dunedin 9054, New Zeland}
\email{lclark@maths.otago.ac.nz}

\author[D. Mart\'{\i}n]{Dolores Mart\'{\i}n Barquero}
\address{D. Mart\'{\i}n Barquero: Departamento de Matem\'atica Aplicada, Escuela T\'ecnica Superior de Ingenieros Industriales, Universidad de M\'alaga. 29071 M\'alaga. Spain.}
\email{dmartin@uma.es}

\author[C. Mart\'{\i}n]{C\'andido Mart\'{\i}n Gonz\'alez}
\address{C. Mart\'{\i}n Gonz\'alez:  Departamento de \'Algebra Geometr\'{\i}a y Topolog\'{\i}a, Fa\-cultad de Ciencias, Universidad de M\'alaga, Campus de Teatinos s/n. 29071 M\'alaga. Spain.}
\email{candido@apncs.cie.uma.es}

\author[M. Siles ]{Mercedes Siles Molina}
\address{M. Siles Molina: Departamento de \'Algebra Geometr\'{\i}a y Topolog\'{\i}a, Fa\-cultad de Ciencias, Universidad de M\'alaga, Campus de Teatinos s/n. 29071 M\'alaga.   Spain.}
\email{msilesm@uma.es}

\thanks{
The first author is supported by the Marsden grant 15-UOO-071 from the Royal Society of New Zealand and a University of Otago Research Grant.
The last three authors are supported by the Junta de Andaluc\'{\i}a and Fondos FEDER, jointly, through projects  FQM-336 and FQM-7156. They are also supported by the Spanish Ministerio de Econom\'ia y Competitividad and Fondos FEDER, jointly, through project  MTM2013-41208-P.
\newline
This research took place while the first author was visiting the Universidad de M\'alaga. She thanks her coauthors for their hospitality.}

\begin{abstract} 
Given an arbitrary graph $E$ we investigate the relationship between $E$ and the groupoid $G_E$.
 We show that there is a lattice isomorphism between the lattice of pairs $(H, S)$, where $H$ is a hereditary and saturated set of vertices and $S$ is a set of breaking vertices {associated to $H $}, onto the lattice of open invariant subsets of $\geo$. We use this lattice isomorphism to characterize the decomposability of the Leavitt path algebra $L_K(E)$, where $K$ is a field. 
 First we find a graph condition to characterise when an open invariant  subset of $\geo$ is closed.
 Then we give both a graph condition and a groupoid condition each of which is equivalent to $L_K(E)$ being decomposable {in the sense that it can be written as a direct sum of two nonzero ideals}.
\end{abstract}
\maketitle

\section{Introduction}
Leavitt path algebras were introduced in \cite{AA1} and \cite{AMP} as a generalization of Leavitt algebras (\cite{L3}). They are the algebraic counterpart of operator algebras associated to directed graphs (\cite{R}). They have attracted the attention of both, ring theorists and operator algebraists and the techniques that emerged to study these algebras have typically been algebraic, analytic, or a mix of both. In this paper our techniques are more topological in nature.

In  \cite{CFST}, the authors show that each Leavitt path algebra $L_K(E)$
is isomorphic to a \emph{Steinberg algebra} defined as follows:
first, one constructs a topological groupoid $G_E$ from the directed
graph $E$. Next, the Steinberg algebra associated to 
$E$ is the convolution algebra of locally constant, compactly supported
functions from $G_E$ into the field $K$.

In this paper, we use the Steinberg algebra
model of a Leavitt path algebra to determine its algebraic decomposability.
An algebra $A$ is called decomposable if it can be written as the direct sum of two nonzero ideals. For Leavitt path algebras it turns out that decomposability is equivalent to graded decomposability (see Lemma \ref{lem.equivDec}).

In our main theorem  we show that the decomposability of $L_K(E)$ is equivalent to the topological decomposability of $G_E$. Moreover, we show that $L_K(E)$
 is decomposable if and only if
there exists a nonempty, proper,  hereditary and saturated subset $H$ of $E^0$ such
that two additional graph conditions are satisfied: 
\begin{enumerate}
\item[(a)] Every infinite path whose vertices are outside of $H$ eventually does not connect to
$H$.
\item [(b)] Every infinite emitter having an infinite number of edges connecting to $H$  must be either in $H$ or a breaking vertex for $H$. 
\end{enumerate}
\noindent
(See the conditions in  Theorem \ref{thm:decomp} (\ref{it3:decomp}).)

Before proving our main result (Theorem \ref{thm:decomp}), we show how some graph properties and groupoid properties are connected. Consider the
lattice $\te$ of pairs $(H,S)$ where $H$ is a hereditary saturated 
subset of vertices and $S$ is a set of breaking vertices of $H$. We are interested in $\te$ because this lattice is isomorphic to the lattice of graded ideals of the Leavitt path algebra (\cite[Theorem 2.5.6]{AAS}).

We establish a lattice isomorphism from $\te$ onto
the lattice $\oe$ of open invariant  subsets of $\geo$ (see Theorem~\ref{thm:Lattice}).
This result generalises \cite[Lemma 6.4]{CEAaHS} to graphs that are not necessarily row-finite.
We remark that for row-finite graphs without sinks, Lemma 6.4 in \cite{CEAaHS} says that there is an isomorphism from the lattice of hereditary and saturated subsets of vertices onto $\oe$, which is not surprising.
 However when infinite emitters are present the appropriate lattice  to be considered is $\te$, which involves not only  hereditary and saturated sets but also to breaking vertices, while the lattice $\oe$ remains the same.
 
We find our lattice isomorphism from $\te$ to $\oe$ very useful:  The lattice structure of $\te$ is complicated.
In particular, the explicit formula for the $\vee$ relation is notoriously difficult to define.
However, the lattice structure of $\oe$ is the standard one $(\subseteq, \cup, \cap)$. Moreover, we prove that there is also a lattice isomorphism between $\oe$ and that of the graded ideals of  the Steinberg algebra $A_K(G_E)$ (see Corollary \ref{cor:LatticeIsom}).

The last ingredients we need to prove Theorem \ref{thm:decomp} are necessary and sufficient
conditions on a pair $(H,S)$ to ensure the corresponding element
$U_{H,S} \in \oe$ is closed, which we give in
Proposition~\ref{prop:closed}. In fact, the set $S$ of breaking vertices is completely determined as it must be the set of all breaking vertices for the hereditary and saturated set $H$.

Leavitt path algebra decomposability is also studied in \cite{AN}. In the final section we compare our results to those in \cite{AN}. 
We expected our decomposibility conditions to be equivalent to the characterisation given
in \cite[Theorem~4.2]{AN}.  However, we find that
\cite[Theorem~4.2]{AN} is incorrect; see Example 
\eqref{exm:error} for a counterexample.  
In Theorem \ref{thm:nuevaprueba} we provide a modified formulation which also appears in \cite[Proposition~4.5]{AN}
 and \cite[Proposition~4.6]{AN}.
We then show the equivalence between the graph conditions in our Theorem \ref{thm:decomp} and the graph condition in Theorem \ref{thm:nuevaprueba}.

%%%%%%%%%%%%%%%%%%%%%%%%%%%%%%
%%%%%%%%%%%%%%%%%%%%%%%%%%%%%%
\section{Preliminaries}

\noindent \textbf{Directed Graphs.}   A \emph{directed graph} is a 4-tuple $E=(E^0, E^1, r_E, s_E)$ 
consisting of two disjoint sets $E^0$, $E^1$ and two maps
$r_E, s_E: E^1 \to E^0$. The elements of $E^0$ are the \emph{vertices} and the elements of 
$E^1$ are the edges of $E$.  Further, for $e\in E^1$, $r_E(e)$ and $s_E(e)$ are 
called the \emph{range} and the \emph{source} of $e$, respectively. 
If there is no confusion with respect to the graph we are considering, we simply write $r(e)$ and $s(e)$.

A vertex which emits no edges is called a \emph{sink}. 
A vertex $v$ is called an \emph{infinite emitter} if $s^{-1}(v)$ is an infinite set, and a \emph{regular vertex} if it is neither a sink nor an infinite emitter. 
The  set of infinite emitters will be denoted by ${\rm Inf}(E)$ while ${\rm Reg}(E)$ will denote the set of regular vertices.

In order to define the Leavitt path algebra, we need to introduce the  
{\it extended graph of} $E$.  This is the graph 
$\widehat{E}=(E^0,E^1\cup (E^1)^*, r_{\widehat{E}}, s_{\widehat{E}}),$ where
$(E^1)^*=\{e_i^* \ | \ e_i\in  E^1\}$ and the functions $r_{\widehat{E}}$ and $s_{\widehat{E}}$ are defined as 
\[{r_{\widehat{E}}}_{|_{E^1}}=r,\ {s_{\widehat{E}}}_{|_{E^1}}=s,\
r_{\widehat{E}}(e_i^*)=s(e_i), \hbox{ and }  s_{\widehat{E}}(e_i^*)=r(e_i).\]
The elements of $E^1$ will be called \emph{real edges}, while for $e\in E^1$ we will call $e^\ast$ a
\emph{ghost edge}.    

A nontrivial \emph{path} $\mu$ in a graph $E$ is a finite sequence of edges $\mu=\mu_1\dots \mu_n$
such that $r(\mu_i)=s(\mu_{i+1})$ for $i=1,\dots,n-1$.\footnote{We warn the reader that this convention
has not been universally adopted and some of our references use the path convention of \cite{R}, where $s(\mu_i)=r(\mu_{i+1})$.}
In this case, $s(\mu):=s(\mu_1)$ and $r(\mu):=r(\mu_n)$ are the
\emph{source} and \emph{range} of $\mu$, respectively, and $n$ is the \emph{length} of $\mu$. We also say that
$\mu$ is \emph{a path from $s(\mu_1)$ to $r(\mu_n)$} and denote by $\mu^0$ the set of its 
vertices, i.e.,
$\mu^0:=\{s(\mu_1),r(\mu_1),\dots,r(\mu_n)\}$. By convention $r(\mu_0):=s(\mu)$.

We view the elements of $E^{0}$ as paths of length $0$. The set of all (finite) paths of a graph $E$ is denoted by ${\rm Path}(E)$.
In this paper we
will also consider infinite paths, i.e., infinite sequences $x=x_1x_2\dots$, where $x_i\in E^1$ and 
$r(x_i)=s(x_{i+1})$ for every $i\in \N$.   We denote the set of all infinite paths by $E^{\infty}$. We also define $s(x):=s(x_1)$.

For vertices $u, v$ we say $u \ge v$ whenever there is a path $\mu$ such that $s(\mu)=u$ and $r(\mu)=v$, 
A subset $H$ of $E^0$ is called \emph{hereditary} if $v\ge w$ and $v\in H$ imply $w\in H$. A
hereditary set is \emph{saturated} if every regular vertex which feeds into $H$ and only into $H$ is again
in $H$, that is, if $s^{-1}(v)\neq \emptyset$ is finite and $r(s^{-1}(v))\subseteq H$ imply $v\in H$. 

Let $E$ be a graph. For every non empty hereditary subset $H$ of $E^0$, define
$$F_E(H)=\{ \a\in {\rm Path}(E) \mid \a_i\in
E^1, s(\a_1)\in E^0\setminus H, r(\a_i)\in E^0\setminus H \mbox{ for } i<\vert \alpha\vert, r(\a_{\vert \alpha\vert})\in H\}.$$ 

\medskip

\noindent \textbf{Leavitt path algebras.}   Given a (directed) graph $E$ and a field $K$, the {\it path $K$-algebra} of $E$,
denoted by $KE$, is defined as the free associative $K$-algebra generated by the
set of paths of $E$ with relations:
\begin{enumerate}
\item[(V)] $vw= \delta_{v,w}v$ for all $v,w\in E^0$.
\item [(E1)] $s(e)e=er(e)=e$ for all $e\in E^1$.
\end{enumerate}
The {\it Leavitt path algebra of} $E$ {\it with coefficients in} $K$, denoted $L_K(E)$, 
is the quotient of the path algebra $K\widehat{E}$ by the ideal of $K\widehat{E}$ generated by the relations:
\begin{enumerate}    
\item[(CK1)] $e^*e'=\delta _{e,e'}r(e) \ \mbox{ for all } e,e'\in E^1$.
\item[(CK2)] $v=\sum _{\{ e\in E^1\mid s(e)=v \}}ee^* \ \ \mbox{ for every}\ \ v\in  {\rm Reg}(E).$
\end{enumerate}
Observe that in $K\widehat{E}$ the relations (V) and (E1) remain valid and that the following is also satisfied:
\begin{enumerate}
\item [(E2)] $r(e)e^*=e^*s(e)=e^*$ for all $e\in E^1$.
\end{enumerate}
\medskip

One can show
\[L_K(E) = \operatorname{span} \{\alpha\beta^{\ast } \ \vert \ \alpha, \beta \in {\rm Path}(E)\}\]
 and that $L_{K}(E)$ is a $\mathbb{Z}$-graded $K$-algebra. 
In particular, for each $n\in\mathbb{Z}$, the degree $n$ component $L_{K}(E)_{n}$ is spanned by the set  
\[\{\alpha \beta^{\ast }\ \vert \  \alpha, \beta \in {\rm Path}(E)\ \hbox{and}\  \mathrm{length}(\alpha)-\mathrm{length}(\beta)=n\}.\] 
%Denote by $h(L_K(E))$ the set of all homogeneous elements in $L_K(E)$, that is,  $$h(L_K(E)):= \cup_{n\in\mathbb{Z}}L_{K}(E)_{n}.$$

Our notation for elements of $L_K(E)$ is standard.  However, in this paper, we often consider elements 
of the set ${\rm Path}(E)$.  In particular,  we will not assume any of the structure that comes with elements of the quotient 
$L_K(E)$.  We hope this will not be too confusing and we try to make clear when
we are working in $L_K(E)$.

\medskip

\noindent\textbf{Groupoids and Steinberg algebras.}  A groupoid $G$ is a generalisation of a group in which the `binary operation' is only partially defined. For a 
precise definition see \cite{Renault}.  The \emph{unit space} of $G$ is denoted $G^{(0)}$  and is defined such that
\[G^{(0)} := \{\gamma \gamma^{-1} \ \vert \ \gamma \in G\} = \{\gamma^{-1}\gamma \ \vert \ \gamma \in G\}.\]
Groupoid \emph{source} and \emph{range} maps $r,s:G \to G^{(0)}$ are defined such that
\[
 s(\gamma) = \gamma^{-1} \gamma \quad \text{ and } \quad r(\gamma) = \gamma \gamma^{-1}.
\]
Although the notation and terminology of $r$ and $s$ have both a graphical and a groupoid interpretation,
it will be clear from context which one we mean.

Steinberg algebras, introduced in \cite{Steinberg}, are algebras associated to `{ample}' groupoids and 
are defined 
as follows.   First, a \emph{topological groupoid} is a groupoid equipped with a topology such that
composition and inversion are continuous.  Next, we say an open subset $B$ of a groupoid $G$ is 
an \emph{open bisection}
if $r$ and $s$ restricted to $B$ are homeomorphisms onto  an open subset of $G^{(0)}$.  
Finally, we say a topological groupoid $G$ is \emph{ample} if there is a basis for its topology 
consisting of compact open bisections.

Given an ample groupoid $G$ and a field $K$, the \emph{Steinberg algebra associated to} $G$ is
\[A_K(G) := \operatorname{span} \{1_B \ \vert \ B \text{ is a compact open bisection}\}\] 
where $1_B:G \to K$ is the characteristic function of $B$.  Addition and scalar multiplication 
in $A_K(G)$ are defined pointwise.  Multiplication is given by the convolution 
$$f\ast g(\gamma)= \sum_{\alpha\beta=\gamma}f(\a)g(\beta).$$

For compact open bisections $B$ and $D$ the convolution gives $1_B 1_D = 1_{BD}.$ In this paper we will omit the $\ast$ symbol for the convolution.

In the remainder of this section, we recall how the Leavitt path algebra is isomorphic to a 
Steinberg algebra.  
The following construction of a groupoid $G_E$ from an arbitrary graph $E$ can be found 
in \cite{Pat2002} and is a generalisation of the construction found in \cite{KPRR} 
in that it allows  graphs with infinite emitters and/or sinks.  Note that  \cite{Pat2002} considers only
countable graphs.  Therefore
we follow the construction of \cite[Example~2.1]{CS} which does not require  countability. 
Define
\begin{align*}X &:= E^\infty \cup \{\mu \in {\rm Path}(E) \mid r_E(\mu)\text{ is a sink}\} \cup 
	\{\mu \in {\rm Path}(E) \ \vert \ r_E(\mu) \in \operatorname{Inf}(E)\} \text{ and }\\
G_E &:= \{(\alpha x, |\alpha| - |\beta|, \beta x) \ \vert \ \alpha,\beta \in {\rm Path}(E),
    x \in X, r_E(\alpha) = r_E(\beta) = s_E(x)\}.\end{align*}
A pair of elements in $G_E$ is composable if and only if they are of the form $((x,k,y),(y,l,z))$ and then
the composition and inverse maps are defined such that
\[
    (x,k,y)(y,l,z) := (x, k+l, z) \quad\text{ and }\quad (x,k,y)^{-1} := (y,-k,x).
\]
Thus
\[\geo =\{(x, 0, x) \ \vert \ x \in X\} \text{ which we identify with $X$}.\]

 \begin{remark}\label{rmk:completion}
\rm
For any $v\in E^0$ there exists $y_v\in \geo$ such that $s(y_v)=v$. Indeed, if $v$ is an infinite emitter or a sink, then $y_v=v$. Otherwise $v\in {\operatorname{Reg}}(E^0)$. Fix $e_1\in s^{-1}(v)$. If $r(e_1)$ is an infinite emitter or a sink, then we may choose $y_v=e_1$. If it is not, then $r(e_1)\in {\operatorname{Reg}}(E^0)$ and we repeat the same process in order to set $y_v=e_1e_2\dots$ as claimed.
\end{remark}

Next we show how $G_E$ can be viewed as an ample groupoid. For $\mu \in {\operatorname{Path}}(E)$ define
\[
    Z(\mu):= \{\mu x \mid x \in X, r(\mu)=s(x)\} \subseteq X
\] and for $\mu \in {\rm Path}(E)$ and a finite $F \subseteq s^{-1}(r(\mu))$, define
\[
Z(\mu\setminus F) := Z(\mu) \cap \Big(\geo \setminus \bigcup_{\alpha \in F} Z(\mu\alpha)\Big).
\]

Then the sets of the form $Z(\mu \setminus F)$ are a basis of compact open sets for a 
Hausdorff topology on $X = \geo$ by \cite[Theorem~2.1]{Webster:xx11}.
\medskip

We point out that the topological structure around infinite emitters can be surprising. For example, assume that $v$ is an infinite emitter and that $\{e_n\}\subseteq s^{-1}(v)$ is an infinite set. Then 
 $e_ny_{r(e_n)} \to v$.
\medskip

Now for each $\mu,\nu \in {\rm Path}(E)$ with $r(\mu) = r(\nu)$, 
and finite $F \subseteq {\rm Path}(E)$ such that $r(\mu)=s(\alpha)$ 
for all $\alpha \in F$, define
\[
Z(\mu,\nu) := \{(\mu x, |\mu| - |\nu|, \nu x) \ \vert \ x \in X, r(\mu)= s(x)\},
\]
and then
\[
Z((\mu,\nu) \setminus F) := Z(\mu,\nu) \cap
            \Big(\geo\setminus\bigcup_{\alpha \in F} Z(\mu\alpha,\nu\alpha)\Big).
\]
Now the collection $Z((\mu,\nu) \setminus F)$ forms a basis of 
compact open bisections that generate a topology such that
$G_E$ is a Hausdorff ample groupoid.  
\bigskip

Finally, putting everything together we have a graded isomorphism 
\begin{equation}
 \label{def:pi} 
 \pi:L_K(E) \to A_K(G_E) \text{ such that }  \pi(\mu\nu^\ast-\sum_{\alpha\in F}\mu\a\a^\ast\nu^\ast) =1_{Z((\mu, \nu)\setminus F)}
\end{equation}

\noindent 
as shown in  \cite[Example 3.2]{CS}.

\bigskip

%%%%%%%%%%%%%%%%%%%%%%%%%%%%%%%%%%%%%%%%%%%5

\bigskip
 
\section{Lattices} 

In this paper, one thing we focus on is the topological structure of \emph{invariant} subsets of $\geo$. 
We say that $U\subseteq \geo$ is invariant if $x\in U$ and $(x, k, y)\in G_E$ then $y\in U$. Equivalently, $U$ is invariant if $y\in U$ and $(x, k, y)\in G_E$ then $x\in U$. Note that if $U$ is invariant then $G_E\setminus U$ is also invariant. Let $\oe$ denote the set of open invariant subsets of $\geo$. Then $\oe$ is a lattice with the standard structure given by $(\subseteq, \cup, \cap)$.

\begin{remark}\label{rmk:shiftEquiv}
\rm 
Fix an invariant $U\subseteq \geo$. Let $x$ and $y$ be in $\geo$ such that $x$ is shift
equivalent to $y$ (i.e. $x$ and $y$ are eventually the same). Then $x\in U$ if and only if $y\in U$.
\end{remark}

In this section we show that there is a lattice isomorphism from the graph lattice $\te$ (defined below) onto the lattice $\oe$. 
Following \cite[Definitions 2.4.3]{AAS}, 
let $H$ be a hereditary subset  of $E^0$, and let $v\in E^0$.  We say that $v$ is a  \index{breaking vertex} \emph{breaking vertex of} $H$ if $v$ belongs
to the set
$$\index{$B_H$} B{_H}:=\{v\in E^0\setminus H \ \vert \ v\in {\rm Inf}(E) \ \text{and} \ 0 < \vert s^{-1}(v) \cap r^{-1}(E^0\setminus H)\vert < \infty\}.$$
In words, $B_H$ consists of those vertices which are infinite emitters, which do not belong to $H$, 
and for which the ranges of the edges they emit are all, except for a finite (but nonzero) number, inside $H$.
We set 
\[ \te:= \{(H,S) \ \vert \ H \text{ is satuarated and hereditary and } S \subseteq B_H\}.\]
These pairs were also considered in a $C^*$-context, for example in \cite{HS, DT}. 
We define the following relation on $\te$: 
\[ 
(H_1, S_1)\leq (H_2, S_2)\ \text{ if and only if }
H_1 \subseteq H_2 \ \ \text{and} \ \ 
S_1\subseteq S_2\cup H_2.
\]
Then $\te$ has a lattice structure with explicit formulas for $\vee$ and $\wedge$ given in \cite[Remark~3.4]{AN} (see also 
\cite[Proposition 2.5.4]{AAS}).
Note that every graded ideal in a Leavitt path algebra is generated by a pair $(H, S)$ as defined in  
\cite[Definition 2.5.3]{AAS} (see Corollary~\ref{cor:LatticeIsom}). 
  
\begin{definitions}\label{def:USub}
\rm
Let $(H,S)\in\te $. Then define
\begin{align}
 U_{H}&:=\{x\in\geo \ \vert\ r(x_n)\in H \text{ for some } n\},\label{UH}\\
 U_S& :=\{\alpha\in {\rm Path}(E)\ \vert\ r(\alpha)\in S \} \label{US} \text{ and}\\
 U_{H,S}&:=U_H\sqcup U_S. \label{UHS}\\
\intertext{On the other hand, given $U\in\oe$, define}
H_U&:=\{v\in E^0\ \vert\ Z(v)\subseteq U\}\label{HU} \text{ and } \\
S_U&:=\{r(\alpha)\ \vert\ \alpha\in U\cap {\rm Path}(E),\  r(\alpha)\text{ is an infinite emitter and } 
r(\alpha)\not\in H_U \} \subseteq U.\label{SU}
\end{align}
\end{definitions}
\begin{theorem}\label{thm:Lattice}
 Suppose $E$ is an arbitrary graph. 
Then the map $\phi:  {\mathcal T}_E  \to {\mathcal O}_E$ given by:
$$\phi ((H,S)) =   U_{H,S},$$
is a lattice isomorphism whose inverse is the map $ \rho: {\mathcal O}_E \to  {\mathcal T}_E$ given by
$$\rho(U)= (H_U, S_U).$$
 \end{theorem}

We  will give the proof in a series of lemmas.
\begin{lemma}\label{lem:OpInv}
Fix $(H,S)\in \te$,
then  $U_{H,S}\in\oe$. 
\end{lemma}  
\begin{proof}
First we prove that $U_{H,S}$ is open. Fix $x\in U_{H,S}$. If $x\in U_H$ then there exists an $n$ such that $r(x_n)\in H$.
Thus, $$x\in Z(x_1\cdots x_n)\subseteq U_H$$
because $H$ is hereditary. Now in the  case where  $x \in U_S$, define $$F:=\{e\in E^1\ \vert \ s(e)=r(x) \text{ and }
r(e)\notin H\}$$ 
which is finite by the definition of $S$. Then $x\in Z(x\setminus F)\subseteq U_{H,S}$.  

Next we show that $U_{H,S}$  is invariant. Consider $(\alpha x, k,\beta x)\in G_E$ such that $\beta x\in \uhs$. If $\beta x\in U_H$ then
since $H$ is hereditary, there exists some $n\in\N\cup\{0\}$  such that $r(x_n)\in H$. This implies that $\alpha x\in U_H$.  

On the other hand, if $\beta x\in U_S$,  then $r(\beta x)$ is an infinite emitter and equals  $r(\alpha x)$ so $\alpha x\in U_S$.
\end{proof}

\begin{lemma}
Fix $U\in\oe$, then $(H_U,S_U)\in\te$.
\end{lemma}
\begin{proof}
First we show that $H_U$ is hereditary. Fix $e\in E^1$ such that $s(e)\in H_U$. To see that $r(e)\in H_U$ it suffices to prove $Z(r(e))\subseteq U$. So fix $x\in Z(r(e))$. Then 
$x\in U$  because $U$ is invariant. 

Next we show $H_U$ is saturated. Fix a regular vertex $v$ such that $r(s^{-1}(v))\subseteq H_U$.
We must show $v\in H_U$, that is, $Z(v)\subseteq U$. In order to do this fix $x\in Z(v)$. We may write
$x=ey$  for some $e\in s^{-1}(v)$. Thus, $r(e)\in H_U$. Thus $Z(r(e))\subseteq U$. But now we have
$y\in Z(r(e))\subseteq U$. Therefore $x\in U$ again by invariance of $U$.

Finally we show $S_U\subseteq B_{H_U}$. Fix $v\in S_U$.  Note that $U$ invariant implies $v\in U$. 
Since $U$ is open, there exists a  finite $F\subseteq {\rm Path}(E)$ such that $Z(v\setminus F)\subseteq U$.
Define now $$F_1:=\{\alpha_1 \ \vert \ \a=\a_1\cdots\a_n \in F\}.$$
We claim that $$r(s^{-1}(v))\cap (E^0\setminus H_U)\subseteq r(F_1).$$ 
To prove the claim, fix a $w\in r(s^{-1}(v))\cap (E^0\setminus H_U)$. Then there is some $e\in E^1$ such
that $s(e)=v$ and $r(e)=w\not\in H_U$. We know that $Z(e)\subseteq Z(v)$ but since $w\notin H_U$ then $Z(w)\nsubseteq U$ and by invariance $Z(e)\nsubseteq U$ either.
Because $Z(v\setminus F)\subseteq U$, we must have $$Z(e)\cap(\bigcup_{\a\in F} Z(\a)) \neq\emptyset.$$
Consequently there is some $\a\in F$ such that $Z(e)\cap Z(\a)\neq \emptyset$ and in fact we must
have $\alpha_1=e$ and hence $w\in r(F_1)$, proving the claim. Since $F$ is finite, we have the finiteness of 
$r(s^{-1}(v))\cap (E^0\setminus H_U)$.  This, together with the fact that $v$ is
an infinite emitter implies $v\in B_{H_U}$. 
\end{proof}

\begin{lemma}\label{lem:surjectivity}
We have
$\phi\circ\rho=1_{\oe}$.
\end{lemma}
\begin{proof}
Fix $U\in\oe$. Then 
$$\phi (\rho(U))= \phi((H_U, S_U))= U_{{H_U}, S_U}.$$ 
We show that $ U_{{H_U}, S_U}=U$. Starting with the left hand side, fix $x\in  U_{{H_U}, S_U}$. 

Case 1. Suppose $x\in U_{H_U}$. Then, there exists $n\in \N$ such that $r(x_n)\in H_U$. Thus $Z(r(x_n))\subseteq U$ and hence $x\in U$ by Remark \ref{rmk:shiftEquiv}.

Case 2. Suppose $x\in U_{S_U}$. Then $r(x)\in S_U \subseteq U$.  Consequently, $x\in U$ by Remark \ref{rmk:shiftEquiv}.

We have proved $U_{{H_U}, S_U}\subseteq U$. For the reverse inclusion, fix $x\in U$. We consider the three possibilities for $x$.

Case 1. Assume $x\in {\rm Path}(E)$ and $r(x)$ is a sink. By Remark \ref{rmk:shiftEquiv} we have that $r(x) \in U$. Now 
$$\{r(x)\}= Z(r(x))\subseteq U$$
and this implies $r(x)\in H_U$ giving $x\in U_{H_U}$.

Case 2. Suppose $x\in E^\infty$. Since $U$ is open, there exists $n\in \N$ such that 
$$Z(x_1 \dots x_n)\subseteq U.$$
Since each element in $Z(r(x_n))$ is shift equivalent to an element in $Z(x_1 \dots x_n)$ we get    $Z(r(x_n))\subseteq U$ by Remark \ref{rmk:shiftEquiv}. This means $r(x_n)\in H_U$. Then $x\in U_{H_U}$ by Definition \ref{def:USub} (\ref{UH}).

Case 3. Finally, assume $x\in {\rm Path}(E)$ and $r(x)$ is an infinite emitter. 
If $r(x) \in H_U$ then $x\in U_{H_U}$ by Definition \ref{def:USub} (\ref{UH}). Otherwise, suppose $r(x) \notin H_U$. Then $r(x)\in S_U$  by Definition \ref{def:USub} (\ref{SU}). Now $x\in U_{S_U}$ by Definition \ref{def:USub} (\ref{US}).

In all three cases, $x\in U_{{H_U}, S_U}$ and we are done.
 \end{proof}

\begin{lemma}\label{lem:injectivity}
We have
$\rho\circ\phi=1_{\te}$.
\end{lemma}
\begin{proof}
Fix $(H, S)\in \te$.
Then 
$$\rho (\phi((H, S)))= \rho(U_{H, S})= (H_V, S_V),$$
where $V:=U_{H,S}$.

We show that $ H=H_V$ and $S=S_V$. Fix $v\in H$. To see that $v \in H_V$ it suffices to show $Z(v)\subseteq V$ by Definition \ref{def:USub} (\ref{HU}). Fix $x\in Z(v)$. Then $s(x)=v\in H$ so $x\in U_H$ by Definition \ref{def:USub} (\ref{UH}). Thus $x\in V$.

Next, fix $w\in H_V$. By  Definition \ref{def:USub} (\ref{HU}) we have $Z(w)\subseteq V$. By way of contradiction, assume $w\notin H$. We will show that there exists an infinite path $x$ such that $s(x)=w$ and $r(x_n)\notin H$ for every $n\in \N$. Notice that $w$ is not a sink. For otherwise, $w\in Z(w)\subseteq V$ but $w\notin U_S$ so $w\in U_H$; but then $w\in H$, a contradiction. 
 We next claim that there exists $e_1\in E^1$ such that $s(e_1)=w$ and $r(e_1)\notin H$. 
 To prove the claim we will consider the two possible cases for $w$. First, if $w\in {\rm Reg}(E)$ the existence of 
 $e_1$ comes from $H$ being saturated and $w\notin H$. Second, suppose 
 $w \in  {\rm Inf}(E)$.  Then $w\in V$. Since $w\notin U_H$ we must have $w\in U_S$ and therefore $w\in S$. Now the existence of $e_1$ comes from the 
 definition of $B_H$.
Continuing in this way by replacing $w$ with $r(e_1)$ and noting that 
$Z(r(e_1)) \subseteq V$  by Remark~\ref{rmk:shiftEquiv}, 
we get an infinite path $x=e_1e_2\dots$ such that $r(e_n)\notin H$  for all $n\in \N$. 
 
Combining  $Z(w)\subseteq V$ and $x\in E^\infty\cap Z(w)$ 
gives $x\in U_H$, contradicting the definition of $U_H$.  Therefore $w\in H$ and we have proved $H=H_V$.

To see that $S=S_V$ notice
$$S=\{r(\alpha) \ \vert \ \alpha \in U_S\}.$$
Taking into account $H=H_V$ we have
$$S_V:=\{ r(\alpha) \ \vert \ \alpha \in V\cap {\rm Path}(E), r(\alpha) \in {\rm Inf}(E)\setminus H\}=S.$$
\end{proof}

\begin{lemma}\label{lem:OrderPreserv}
The maps $\phi$ and $\rho$ are both order preserving.
\end{lemma}
\begin{proof}
Consider $(H_1, S_1), (H_2, S_2)\in \te$ such that $(H_1, S_1) \leq(H_2, S_2)$. Notice that by Definition \ref{def:USub}(\ref{UH}) we obtain $U_{H_1}\subseteq U_{H_2}$. S
imilarly, by Definition~\ref{def:USub}(\ref{US}) we get $U_{S_1}\subseteq U_{S_2}\cup U_{H_2}$ and that $\phi$ is order preserving.

To see that $\rho$ is order preserving, assume $U, V\in \oe$ such that $U\subseteq V$. That $H_U\subseteq H_V$ follows immediately from  Definition  \ref{def:USub}(\ref{HU}). 
Now we show $S_U \subseteq S_V \cup H_V$. Take $r(\alpha)\in S_U$. If $r(\alpha)\in H_V$ we are done. Otherwise $r(\alpha)\in S_V$ because $\alpha \in U\subseteq V$.
\end{proof}

\begin{proof}[Proof of Theorem \ref{thm:Lattice}]
The result follows from the lemmas above and \cite[Theorem 8.2]{Jacob} which says that an order preserving
bijection between lattices that has an order preserving
inverse is a lattice isomorphism.
\end{proof}

Next, we present a corollary connecting our map $\phi$ to the lattice structure of graded ideals in the algebras.  
First some notation:  for any algebra $A$, denote by $\lgr(A)$ the lattice of graded ideals of $A$.

For any open invariant $U\in \geo$ we define the map $I$ from $ \oe$ to the lattice of ideals of  $A_K(G_E))$ such that 
\begin{align*}I(U)&:= \{f\in A_K(G_E)\ \vert \ s({\rm \supp} f) \subseteq U\}\\
&=\operatorname{span}\{1_B \ \vert \ B \text{ is a compact open bissection such that } s(B) \subseteq U\}.\end{align*}
Note that $I(U)$ is an ideal of the Steinberg algebra $A_K(G_E)$ by \cite[Lemma 3.2]{CEAaHS}. 
In Corollary~\ref{cor:LatticeIsom} we show that $I$ is a lattice isomorphism onto the graded ideals of $A_K(G_E)$.

Finally, in the Leavitt path algebra, 
for $v\in B{_H}$, define \index{$v^H$} $$v^H:= v-\sum_{e\in s^{-1}(v) \cap r^{-1}(E^0\setminus H) }ee^\ast,$$ and, for any subset $S\subseteq B{_H}$, define
 \index{$S^H$} $S^H=\{v^H \ \vert \ v\in S\}$.
 {For any subset $X$ in an algebra, we denote by $\operatorname{Ideal}(X)$
% We write $\operatorname{Ideal}(H \cup S^H)$ to denote the ideal in $L_K(E)$ generated by $H \cup S^H$.
the ideal generated by $X$.}

\begin{corollary}\label{cor:LatticeIsom}
Consider the following diagram 
$$\xymatrix{ 
\lgr(L_K(E))     \ar[rr]^{\pi} &                                                       & \lgr(A_K(G_E))           &                              \\ 
                                                                          &  &                                                           &   \\
\te \ar[rr]_{\phi}     \ar[uu]^{\varphi'}                    &                           & \oe    \ar[uu]_{I}                                                      &
}
$$
\noindent
where $\phi$ is the map defined in Theorem \ref{thm:Lattice}, $\pi$ is as defined in (\ref{def:pi}) and
 \[\varphi'((H, S))=\operatorname{Ideal}(H\cup S^H).\]  Then the diagram is commutative and all the maps are lattice isomorphisms.
\end{corollary}
\begin{proof}
Since $\phi, \varphi'$ and $\pi$ are lattice isomorphisms it 
suffices to show that \[(I\circ\phi)((H, S))= (\pi \circ\varphi')((H, S))\] for all $(H, S)\in \te$.

Fix $(H, S)\in \te$. For $w\in S$ define $F_w:=\{e\in s^{-1}(w) \cap r^{-1}(E^0\setminus H)\}$.
We have $(I\circ\phi)((H, S))= I(U_{H, S})$ and
\begin{align*}
(\pi \circ\varphi')((H, S)) & =  \pi(\operatorname{Ideal}(H\cup S^H))\\
 & =\operatorname{Ideal}\left(\bigcup_{ v \in H}\{1_{Z(v)}  \}\cup\left(\bigcup_{w\in S}\left\{1_{Z(w)}-\sum_{e\in F_w}{1_{Z(e, r(e))}1_{Z(r(e), e))}}\right\}\right)\right) \\
 & =\operatorname{Ideal} \left(\bigcup_{ v \in H}\{1_{Z(v)}  \}\cup\left(\bigcup_{w\in S}\left\{1_{Z(w)}-\sum_{e\in F_w}{1_{Z(e)}}\right\}\right)\right) \\
 & = \operatorname{Ideal}\left(\bigcup_{ v \in H}\{1_{Z(v)}  \}\cup\left(\bigcup_{w\in S}\left\{1_{Z(w\setminus F_w)}\right\}\right)\right). 
\end{align*}

First we show $(\pi \circ\varphi')((H, S)) \subseteq (I\circ\phi)((H, S))$.
Fix an element $v\in H$. Then we have $1_{Z(v)}\in I(U_{H, S})$ because ${\supp}(1_{Z(v)})\subseteq U_H$. Now, fix $w\in S$. 
Then $1_{Z(w\setminus F_w)}\in I(U_{H, S})$ as 
$\supp(1_{Z(w\setminus F_w)})\subseteq U_{H, S}$. Thus $(\pi \circ\varphi')((H, S))\subseteq (I\circ\phi)((H, S))$ as every 
generator of $(\pi \circ\varphi')((H, S))$ belongs to $(I\circ\phi)((H, S))$.

Now we prove $(I\circ\phi)((H, S))\subseteq (\pi \circ\varphi')((H, S))$. 
Fix a generator $1_B \in (I\circ\phi)((H, S))$.  First observe that the collection
\[
\{Z(\alpha)\}_{\alpha\in F_E(H)} \cup \{Z(\alpha\setminus F_{r(\alpha)})\}_{r(\alpha)\in S}\]
forms a disjoint open cover of $U_{H,S}$.  Thus there exist finite subsets $A_H \subseteq F_E(H)$ and 
$A_S \subseteq S$ such that
\[\mathcal{F}:=\{Z(\a)\}_{\a\in A_H} \cup \{Z(\alpha \setminus F_{r(\alpha)})\}_{r(\alpha) \in A_S}\]
is a disjoint finite subcover of the compact set $s(B)$.  
We claim that 
\[f:=\sum_{C \in \mathcal{F}} 1_C\in  (\pi \circ \varphi')((H,S)).\]
First fix $C$ in $A_H$.
There exists $\a\in F_E(H)$ such that 
$$
C= Z(\a) = Z(\a, r(\a))Z(r(\alpha)).
$$
Since 
$$1_C= 1_{Z(\a))}= 1_{Z(\a, r(\a))Z(r(\alpha))} = 1_{Z(\a, r(\a))} 1_{Z(r(\alpha))}$$
we have
 \[1_C \in \  (\pi \circ \varphi')((H,S))\] 
because $1_{Z(r(\alpha))} \in \  (\pi \circ \varphi')((H,S))$.
\medskip
 
Next, fix 
$C \in \{Z(\alpha \setminus F_{r(\alpha)})\}_{r(\alpha) \in A_S}$.  
So there exists
$\alpha$ such that $r(\alpha) \in S$ and 
 \[C= Z(\alpha \setminus F_{r(\alpha}) = Z(\alpha, r(\alpha))Z(r(\alpha)\setminus F_{r(\alpha)}).\] 
Since \[1_{Z(r(\alpha)\setminus F_{r(\alpha)})} \in (\pi \circ \varphi')((H,S))\] we have
\[
 1_C = 1_{Z(\alpha, r(\alpha))Z(r(\alpha)\setminus F_{r(\alpha)})} = 1_{Z(\alpha,
 r(\alpha))}1_{(Z(r(\alpha)\setminus F_{r(\alpha)})} \in
 (\pi \circ \varphi')((H,S)).
\]
Thus $f \in (\pi \circ \varphi')((H,S))$ as claimed.  Now
\[1_B = 1_B f \in (\pi \circ \varphi')((H,S)).\]
\end{proof}

\section{Clopen invariant sets}

In this section we give necessary and sufficient conditions on a pair $(H, S)\in \te$ to ensure that $U_{H, S}$ is a clopen invariant set. Since $U_{H, S}$ is always an open invariant set (Lemma \ref{lem:OpInv}), we characterize when it is also closed.

Let $E$ be an arbitrary graph. For any subset $X\subseteq E^0$ and any vertex $v$, we denote by
${\rm Path}(v, X)$ the set of all finite paths $\a$ such that $s(\a)= v$ and $r(\a)\in X$.

\begin{proposition}\label{prop:closed}
Let $(H, S)\in \te$ and let $U_{H, S}$ be as given in Definition \ref{def:USub}(\ref{UHS}). Then, $U_{H, S}$ is closed if and only if the following two conditions are satisfied:
\begin{enumerate}[\rm (i)]
\item If $\ x\in E^\infty$ is such that  $x^0\subseteq E^0\setminus H,$ then there exists $ N \in \N$ such that
\[\operatorname{Path}(r(x_N),H) = \emptyset.\] \label{eq:closed1}
\item If $v \in E^0 $ is such that $ |\{e\in E^1 \ \vert \ s(e) =v \text{ and } \operatorname{Path}(r(e),H) \neq \emptyset\}| = \infty$,
then $v \in H \cup  S.$ \label{eq:closed2}
\end{enumerate}
\end{proposition}
\begin{proof}
Suppose that $U_{H, S}$ is closed. To see that condition~(\ref{eq:closed1}) holds, 
fix $x\in E^\infty$ such that $x^0\subseteq E^0\setminus H$. 
By way of contradiction, suppose for each $n\in \N$ we have  $\operatorname{Path}(r(x_n), H) \neq \emptyset$. 
Choose $(\alpha)_n \in \operatorname{Path}(r(x_n),H)$ and consider the sequence:
$$\left( \left(x_1\dots x_n(\alpha)_n y_{r((\alpha)_n)}\right)_n\right)\subseteq U_H,$$
where $y_{r((\alpha)_n)}$   is as in Remark \ref{rmk:completion}. Notice that this sequence converges 
to $x\notin U_{H, S}$, which is a contradiction because $U_{H, S}$ is closed.  
Thus, condition (\ref{eq:closed1}) holds.

To see that condition (\ref{eq:closed2}) holds, fix $v \in E^0$ 
such that 
\[|\{e\in E^1 \ \vert \ s(e) =v \text{ and } \operatorname{Path}(r(e),H) \neq \emptyset\}| = \infty.\]
For each $n \in \mathbb{N}$, choose $e_n\in E^1$ such that $s(e_n) = v$,  
$\operatorname{Path}(r(e_n),H) \neq \emptyset$ and $e_n \neq e_m$ for $m \neq n$.
Also choose $(\alpha)_n \in \operatorname{Path}(r(e_n),H)$.  Now the sequence
$(e_n(\alpha)_n y_{r((\alpha)_n)} )\subseteq U_H$,  where $y_{r((\alpha)_n)}$ is as 
in Remark~\ref{rmk:completion}.  
Then $(e_n(\alpha)_n y_{r((\alpha)_n)}) \rightarrow v$; 
since $U_{H,S}$ is closed, we have $v \in U_{H,S}$.
Therefore $v \in H \cup S$.

Conversely, assume that conditions (\ref{eq:closed1}) and (\ref{eq:closed2})  are satisfied.
Fix a sequence $\left((x)_n\right) \subseteq U_{H,S}$ such that $ (x)_n \rightarrow z$. 
We consider the possible cases for $z$.
 
Case 1.  Suppose $z \in \operatorname{F}(E)$ and $r(z)$ is a sink.  Then $\{z\} = Z(z)$ is an open set and hence
$(x)_n = z$ eventually.  Thus $z \in U_{H,S}$.

Case 2. Suppose $z\in E^\infty$. If $r(z_m)\in H$ for some $m\in \N$, then $z\in U_H$ by 
Definition~\ref{def:USub}(\ref{UH}). Otherwise, suppose $z^0\subseteq E^0\setminus H$. 
By condition~(\ref{eq:closed1}) there exists $N\in \N$ such that  $\operatorname{Path}(r(z_N),H) = \emptyset$.
Notice that $z$ is in the open set $Z(z_1 \dots z_N)$, so we have $(x)_n\in Z(z_1 \dots z_N)$ eventually.
However $Z(z_1 \dots z_N) \cap U_H=\emptyset$, thus $(x)_n\in U_S$ eventually. 
This is a contradiction because $(x)_n \in S$ implies  $\operatorname{Path}(r((x)_n), H) \neq \emptyset$.

Case 3.  Suppose $z \in \operatorname{F}(E)$ and $r(z) \in \operatorname{Inf}(E)$. If $r(z)\in H$ then $z\in U_H$ and we are done. Similarly, if $(x)_n = z$ eventually, then $z\in U_{H, S}$.
So, assume $r(z)\notin H$ and $ (x)_n \neq z$ eventually. Define \[F:= \{e\in E^1 \ \vert \ s(e) =r(z) \text{ and } \operatorname{Path}(r(e),H) \neq \emptyset\}.\]
We claim that  $|F|=\infty$.  To prove the claim, by way of contradiction suppose $|F| < \infty.$
Then $Z(z \setminus F)$ is an open set containing $z$ and hence $(x)_n \in Z(z \setminus F)$ eventually.
Note that for every path $\mu \in Z(z \setminus F)$ with $\mu \neq z$ we have $\operatorname{Path}(r(\mu),H) = \emptyset$.
Thus $Z(z \setminus F) \cap U_H = \emptyset$, which implies $(x)_n \in U_S$ eventually.
But then $\operatorname{Path}(r((x)_n),H) \neq \emptyset$, which is a contradiction.  Therefore $|F| = \infty$.
Now condition~(\ref{eq:closed2}) implies that $r(z) \in H \cup  S$ and 
hence $z \in U_{H,S}$ as needed.
\end{proof}

\begin{remark}\label{rmk:TodoBH}
\rm
Note that for any $(H,S)\in \mathcal{T}_E$  we must have $B_H\subseteq\overline{U_{H,S}}$, where $\overline{(\cdot)}$ denotes the topological closure. Indeed, 
 fix $v\in B_H$. We show that any basic neighbourhood of $v$ has nonempty
intersection with $U_{H,S}$.
Since $v\in \operatorname{Inf}(E)$  it belongs to $G_E^{(0)}$. Consider 
any basic neighbourhood $Z$ of $v$. Then  $v\in G_E^{(0)}$ implies that $Z$ is of the form
$Z(v\setminus F)$. Since $v \in B_H$, there exists an edge $e$ such that $s(e)=v$ and $r(e)\in H$.
Extending $e$ to an element $e y_{r(e)}\in G_E^{(0)}$  we can ensure that $e y_{r(e)}\in Z\cap U_{H,S}$. Thus, 
$Z\cap U_{H,S}\ne\emptyset$ and hence $v\in\overline{U_{H,S}}$.
In particular if $U_{H,S}$ is closed, then $B_H\subseteq U_{H,S}$ and so $B_H=S$.

Notice that condition (\ref{eq:closed2}) of Proposition \ref{prop:closed} implies that $S$ must coincide with $B_H$. 
Thus, if $U_{H, S}$ is closed, then $S=B_H$, which is consistent.
The reverse implication is not true as the first example that follows will show.
\end{remark}

Next, we give some examples to illustrate conditions (\ref{eq:closed1}) and  (\ref{eq:closed2}) in Proposition \ref{prop:closed}.

\begin{examples}
\rm
Consider the graph that follows.
\[\xygraph{
!{<0cm,0cm>;<1.5cm,0cm>:<0cm,1.2cm>::}
!{(0,0) }*+{{\bullet}^{v_1}}="a"
!{(1,0) }*+{{\bullet}^{v_2}}="b"
!{(2,0) }*+{{\bullet}^{v_3}}="c"
!{(3,0)}*+{\cdots}="d"
!{(1.5,.8)}*+{ }
%!{(1.5,0)}*+\xycircle<90pt,15pt>{}
!{(.7,-1.2) }*+{\bullet}="e"
!{(1.3,-1.2) }*+{\bullet}="f"
!{(1.9,-1.2) }*+{\bullet}="g"
!{(1.5,-1.2)}*+\xycircle<70pt,15pt>{}
!{(1.5, -2)}*+{H}
%%%%%%%%%%%%%%%%%
"a":^{e_1}"b":^{e_2}"c":@{..>}"d"
"a":^{f_1}"e" "b":^{f_2}"f" "c":^{f_3}"g"
}\]\vskip 1cm
Then $U_{H, \emptyset}$ is not closed because the sequence $(f_1, e_1f_2, e_1e_2f_3, \dots)$ is contained in 
$U_{H, \emptyset}$ but it converges to $e_1e_2e_3\dots$ which is not in $U_{H, \emptyset}$.  On the other hand, 
condition~\eqref{eq:closed1} fails for the infinite path $e_1e_2e_3 \dots$. 
\medskip

Now, consider the following graph.

\hskip 6cm 
\xygraph{
!{<0cm,0cm>;<1.5cm,0cm>:<0cm,1.2cm>::}
!{(0,0) }*+{\bullet_{u}}="a"
!{(1.5,0) }*+{\bullet_{v}}="b"
%!{(1,-1) }*+{\bullet_{v}}="c"
!{(0,1.5)}*+{}
"a":^{f}"b" 
%"a":_{f}"c"
"a" :@`{"a"+(0,1.5),"a"+(-1.5,-0.3)}_e "a"
}  
\vskip .5cm
Note that, in this case, for $H=\{v\}$ the open set $U_{H, \emptyset}$ contains the sequence $(e^{n}f)$ 
that converges to $ee\dots \notin U_{H, \emptyset}$. So again, $U_{H, \emptyset}$ is not closed.
Here, condition~\eqref{eq:closed1} fails for the infinite path $eee\dots$.
%\vskip -1cm

{Finally, for the graph}

\hskip 7cm 
\xygraph{
!{<0cm,0cm>;<1.5cm,0cm>:<0cm,1.2cm>::}
!{(0,0) }*+{\bullet_{u}}="a"
!{(1.5,0) }*+{\bullet_{v}}="b"
%!{(1,-1) }*+{\bullet_{v}}="c"
!{(0,1.5)}*+{}
"a":^{f_n}_{(\infty)}"b" 
%"a":_{f}"c"
%"a" :@`{"a"+(0,1.5),"a"+(-1.5,-0.3)}_e "a"
}  
\vskip .5cm

\noindent
again we have that the open set $U_{H, \emptyset}$ is not closed, 
where $H=\{v\}$, because the sequence $(f_n)$ converges to $u$, which is not an element of $U_{H, \emptyset}$.
This time, condition~\eqref{eq:closed2} fails for the vertex $u$.
\end{examples}
\bigskip

\section{Decomposability}

Let $K$ be a field and $E$ an arbitrary graph. We say that the Leavitt path algebra $L_K(E)$ is \emph{decomposable} 
if there exists two nonzero ideals $I$ and $J$ in $L_K(E)$ such that $L_K(E)=I \oplus J$.

The equivalence between (\ref{it1:DecLPA})  and  (\ref{it2:DecLPA}) of the following lemma is contained in the proof of \cite[Proposition 4.6]{AN} but is not stated explicitly. We provide a different proof.

\begin{lemma}\label{lem.equivDec}
Let $K$ be a field and $E$ an arbitrary graph. Then the following are equivalent:
\begin{enumerate}[\rm (i)]
\item\label{it1:DecLPA} $L_K(E)$ is decomposable.
\item\label{it2:DecLPA} There exist nonzero graded ideals $I$ and $J$ such that $L_K(E) = I \oplus J$.
\item\label{it3:DecLPA} There exists $(H_1, S_1), (H_2, S_2)\in \te$ such that $L_K(E) = I((H_1, S_1)) \oplus I((H_2, S_2))$.\end{enumerate}
\end{lemma} 
\begin{proof}
(\ref{it2:DecLPA}) implies (\ref{it1:DecLPA}) is trivial. Now assume (\ref{it1:DecLPA}). Suppose $L_K(E)=I \oplus J$, where $I$ and $J$ are not trivial ideals of $L_K(E)$. 
Since $L_K(E^2)= L_K(E)$, then $L_K(E)=I^2 \oplus J^2$ giving $I=I^2$ and $J=J^2$. By \cite[Corollary 2.9.11]{AAS}, 
this is equivalent to saying that the ideals $I$ and $J$ are graded, which implies (\ref{it2:DecLPA}).

Finally (\ref{it2:DecLPA}) and (\ref{it3:DecLPA}) are equivalent by \cite[Theorem 2.5.5]{AAS}.
\end{proof}

\medskip

\begin{theorem} 
 \label{thm:decomp}
 Let $K$ be a field and $E$ be an arbitrary graph.  The following are equivalent:
\begin{enumerate}[\rm (i)]
  \item \label{it1:decomp} $L_K(E)$ is decomposable.
  \item \label{it2:decomp} There exists nonempty clopen invariant subsets $U$ and $V$ in $\geo$ such that 
$A_K(G_E)= I(U) \oplus I(V)$.
  \item \label{it3:decomp}

There exists a nonempty, proper, hereditary and saturated subset $H$ of $E^0$ such that the following conditions are satisfied:
\begin{enumerate}[\rm (a)]
\item If $\ x\in E^\infty$ is such that  $x^0\subseteq E^0\setminus H,$ then there exists $ N \in \N$ such that
\[\operatorname{Path}(r(x_N),H) = \emptyset.\]
\item If $v \in E^0 $ is such that $ |\{e\in E^1 \ \vert \ s(e) =v \text{ and } \operatorname{Path}(r(e),H) \neq \emptyset\}| = \infty$,
then $v \in H \cup  B_H.$
\end{enumerate}

\item \label{it4:decomp} There exists a nonempty clopen invariant subset $U \subsetneq \geo$. 
 \end{enumerate}
\end{theorem} 
\begin{proof}
The equivalence between (\ref{it1:decomp})  and (\ref{it2:decomp}) follows from Lemma \ref{lem.equivDec} and Corollary \ref{cor:LatticeIsom}.
To see that (\ref{it3:decomp}) and (\ref{it4:decomp}) are equivalent, notice that 
if $H$ is as in (\ref{it3:decomp}), then $U_{H,B_H}$ is clopen by Proposition~\ref{prop:closed} giving (\ref{it4:decomp}). Conversely,
if $U$ is as in \eqref{it4:decomp}, then there exists $(H,S) \in \te$ such that $U=U_{H,S}$ by Theorem~\ref{thm:Lattice}.
Since $U$ is closed, $H$ satisfies the conditions of \eqref{it3:DecLPA} by Proposition~\ref{prop:closed}.

It is easy to see that (\ref{it2:decomp})  implies (\ref{it4:decomp}). Lastly,  we show (\ref{it4:decomp})  implies (\ref{it2:decomp}). For $U$ as in  (\ref{it4:decomp}) we claim that we have a decomposition 
 $$A_K(G_E)= I(U) \oplus I(\geo\setminus U).$$
 To see this, fix $f \in A_K(G_E)$ and define
$$f_1(\gamma)=\begin{cases}
f(\gamma) & \text{if}\ s(\gamma) \in U \\
0 & \text{otherwise} 
\end{cases}
\quad
\text{and}
\quad
f_2(\gamma)=\begin{cases}
f(\gamma) & \text{if}\ s(\gamma) \in \geo\setminus U \\
0 & \text{otherwise}. 
\end{cases}
$$ 
 Since $U$ and $\geo\setminus U$ are clopen, $f_1, f_2\in A_K(G_E)$. Observing that $f=f_1 + f_2$ gives the result.
\end{proof}

\section{Compatible Paths}

We finish by comparing our results to those in \cite{AN}. 
We point out that the statement of \cite[Theorem 4.2]{AN} is not correct. Consider the following example.

\bigskip

\begin{example}\label{exm:error}
\rm
Let $E$ be the following graph:
\bigskip

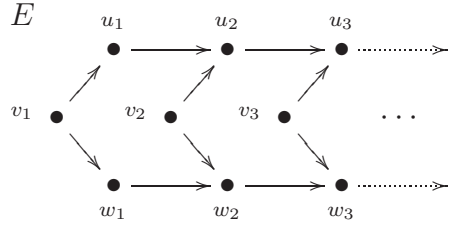
\vbox{
\hskip 5cm
\xygraph{
!{<0cm,0cm>;<1.5cm,0cm>:<0cm,1.2cm>::}
!{(0,0) }*+{\bullet}="a"
!{(0,-0.3) }*+{_{w_1}}
!{(1,0) }*+{\bullet}="b"
!{(1,-0.3) }*+{_{w_2}}
!{(2,0) }*+{\bullet}="c"
!{(2,-0.3) }*+{_{w_3}}
!{(3,0) }*+{}="d"
%%%%%%%%
!{(0,1.5) }*+{\bullet}="e"
!{(0,1.8) }*+{_{u_1}}
!{(1,1.5) }*+{\bullet}="f"
!{(1,1.8) }*+{_{u_2}}
!{(2,1.5) }*+{\bullet}="g"
!{(2,1.8) }*+{_{u_3}}
!{(3,1.5) }*+{}="h"
%%%%%%%
!{(-0.5,0.75)}*+{\bullet}="i"
!{(-0.8,0.75) }*+{^{v_1}}
!{(0.5,0.75)}*+{\bullet}="j"
!{(0.2,0.75) }*+{^{v_2}}
!{(1.5,.75)}*+{\bullet}="k"
!{(1.2,0.75) }*+{^{v_3}}
!{(2.5,0.75)}*+{\cdots}
%%%%%%%
!{(-0.8,1.9)}*+{{E}}
%%%%%%%%%%%%%%%
"a":"b" 
"b":"c"
"c":@{..>}"d"
"e":"f"
"f":"g"
"g":@{..>}"h"
%%%%%%%%
"i":"a"
"i":"e"
"j":"b"
"j":"f"
"k":"c"
"k":"g"
} 
}
\end{example}
\bigskip

Let $X=\{u_n \ \vert \ n \in \N\}$ and $Y=\{w_n \ \vert \ n \in \N\}$. Then the only hereditary and saturated sets are
$X, Y, \emptyset$ and $E^0$. Note that $L_K(E)$ is decomposable as $L_K(E)=I(X)\oplus I(Y)$
 but $X$ and $Y$ do not satisfy the conditions in   \cite[Theorem 4.2]{AN} because there are infinitely many paths starting outside of $X\cup Y$ and ending in $X\cup Y$.
 
\medskip

Theorem \ref{thm:nuevaprueba} below comes from combining the statements of 
\cite[Proposition 4.5]{AN} and  its converse, \cite[Proposition 4.6]{AN}. Here we give an alternative proof by showing the equivalence between the graph condition in 
Theorem \ref{thm:decomp} and the graph condition in   \cite[Proposition 4.5]{AN} 
(that also appears in \cite{Hong}).

First we fix some definitions.

\begin{definition}\rm{(\cite[Definition 4.1]{AN})}
Let $E$ be an arbitrary graph and $H$ a hereditary and saturated set of vertices. A path $\lambda=\lambda_1 \dots \lambda_{\vert \lambda \vert}$, where $\lambda_i\in E^1$, is said to be \emph{$H$-compatible} if $r(\lambda)\in H$ and $s(\lambda_{\vert \lambda \vert})\notin H\cup B_H$.
\end{definition}

\begin{definitions}
\rm
Let $H$ be a hereditary and saturated subset of vertices in a graph $E$ and let $v\in E^0\setminus H$. We define the following sets:
\[
C_{v, H}:= \{\lambda \in {\rm Path}(E) \ \vert \ s(\lambda) = v \ \text{ and } \ \lambda   \text{ is $H$-compatible}\}.
\]
We say that $v$ satisfies \emph{Property} $({\rm P})$ if $\vert C_{v, H} \vert = \infty$.
\end{definitions}

\begin{lemma}\label{lem:compat}
Let $K$ be a field and $E$ an arbitrary graph. Assume that $L_K(E)$ is decomposable and let 
$H$ be a hereditary and saturated subset satisfying conditions (a) and (b) in Theorem \ref{thm:decomp} \eqref{it3:decomp}. If $v\in E^0\setminus H$ 
satisfies Property $({\rm P})$, then there exists $e\in s^{-1}(v)$ such that $r(e)\in E^0\setminus H$ and $r(e)$ satisfies Property $({\rm P})$.
\end{lemma}
\begin{proof}
Decompose $s^{-1}(v)= A \sqcup B$, where $ A:  = \{ e\in s^{-1}(v) \ \vert \ r(e) \in H\}$. We claim that $B\neq \emptyset$. By way of contradiction, assume $B=\emptyset$. We distinguish two cases. If $\vert A \vert = \infty$ then, by condition (b) we have $v\in H\cup B_H$. Since we are assuming $v\notin H$ then $v\in B_H$. But this is a contradiction because $B=\emptyset$. Assume $\vert A \vert \neq \infty$. Note that $\vert A \vert \neq 0$ because $v$ is not a sink. Since $H$ is a saturated set we have $v\in H$, a contradiction. Therefore $B\neq \emptyset$.

Case 1: the set $A$ is infinite. By condition (b) and since $v\notin H$ we have $v\in B_H$. Note that $0\neq \vert B\vert \lneq \infty$. Since $\vert C_{v, H}\vert = \infty$ and for every $e_1\dots e_n\in C_{v, H}$ the edge $e_1$ must be in $B$, then there exists $f\in B$ such that $f\gamma\in C_{v, H}$, for some $\gamma\in {\rm Path}(E)$ and $r(f)$ satisfies Condition (P). In this case we are done.

Case 2: the set $A$ is finite. We are assuming that $\vert C_{v, H} \vert = \infty$. 
%Note that there is not a cycle $c$ based at $v$ because in this case we are contradicting Condition (a). 
Since $v\notin H\cup B_H$, by condition (b) the set of edges which are the first edge of an $H$-compatible path is finite. In this case, we repeat the same argument as before and find $f\in B$ such that $f\gamma\in C_{v, H}$, for some $\gamma\in {\rm Path}(E)$ and $r(f)$ satisfies Condition (P).
Again in this case we are done. This finishes the proof.
\end{proof}

\begin{theorem}\label{thm:nuevaprueba}
Let $K$ be a field and $E$ an arbitrary graph. Then $L_K(E)$ is decomposable if and only if there exist two nontrivial hereditary and 
saturated subsets $H_1$ and $H_2$ 
such that $H_1\cap H_2=\emptyset$ and for every $v\in E^0\setminus (H_1\cup H_2)$, there exists at least one but finitely many paths starting at $v$ which are either $H_1$-compatible or $H_2$-compatible.
\end{theorem}
\begin{proof}
Assume first that there exists a hereditary and saturated subset $H_1$ satisfying  the conditions in Theorem \ref{thm:decomp} \eqref{it3:decomp}. 
Since $U_{H_1, B_{H_1}}$ is a clopen invariant set (by Proposition \ref{prop:closed}), then $\geo\setminus U_{H_1, B_{H_1}}$ is also clopen
invariant and therefore there exits a hereditary and saturated subset $H_2$ such that $$\geo\setminus U_{H_1, B_{H_1}}=U_{H_2, B_{H_2}}$$
by Thereom~\ref{thm:Lattice}. 

Take $v\in E^0\setminus{(H_1\cup H_2)}$. Note that $v$ is not a sink since all of them are in $H_1\cup H_2$. 
We first show that there exists at least one  path starting at $v$ which is $H_1$-compatible  or $H_2$-compatible. 
By way of contradiction, suppose that this is not the case.  We claim that there exists $e\in E^1$ such that $s(e)=v$ and $r(e)\notin H_1\cup H_2$. If $v\in {\rm Reg}(E)$ then there is an edge $e$, starting at $v$, such that $r(e)\notin H_i$   because otherwise $e$ would be an $H_i$ compatible path (for $i=1, 2$). Assume $v$ is an infinite emitter. Then $r(s^{-1}(v)) \not\subseteq H_1$ because in this case every element in  $s^{-1}(v)$ is an $H_1$-compatible path, and we are assuming that this cannot happen. For the same reason $r(s^{-1}(v))\not\subseteq H_2$. If there is a finite number of edges from $s^{-1}(v)$ such that their ranges are inside $H_i$ then each of them is an $H_i$-compatible path, arriving at a contradiction. If $s^{-1}(v)= A \cup B$, where $A$ and $B$ are infinite sets and $r(f)\in H_1$ for every $f\in A$ and $r(f) \in  H_2$ for every $f\in B$ then $f$ is an $H_1$-compatible path (in the first case) or it is an $H_2$-compatible path in the second one, arriving again to a contradiction. In any case we have our claim. 

Fix $e_1$ as before. Reasoning in the same way, for every $n\geq 2$ we can find $e_n\in E^1$ such that $s(e_n)=r(e_{n-1})$ and $r(e_n)\notin H_1\cup H_2$. Define $x=e_1e_2\dots$ and apply that it is in $U_{H_1, B_{H_1}}\sqcup U_{H_2, B_{H_2}}$. 
This implies that $x\in U_{H_1}\sqcup U_{H_2}$, getting a contradiction.
Thus, there exists at least one path starting at $v$ which is either $H_1$-compatible or $H_2$-compatible.

Now we show that the number of paths starting at $v$ which are either 
$H_1$-compatible or $H_2$-compatible is finite. By way of contradiction, and without loss in generality we may assume that $\vert C_{v, H_1}\vert = \infty$. By Lemma \ref{lem:compat} there exists $e_1\in E^1$ such that $s(e_1)=v$, $r(e_1)\notin H$ and $r(e_1)$ satisfies Property (P). Applying the same reasoning to $r(e_1)$ we find $e_2\in E^1$ such that $s(e_2)=r(e_1)$, $r(e_2)\notin H$ and $r(e_2)$ satisfies Property (P). Continuing in this way we find $x=e_1e_2\dots\in E^{\infty}$ such that $x^0\subseteq E^0\setminus H$. Apply condition (a) to obtain that there exists $N\in \N$ such that ${\rm Path}(r(e_N), H)= \emptyset$, a contradiction as every vertex in $x^0$ connects to $H$.

Conversely, assume that there exist hereditary and saturated sets $H_1$ and $H_2$ satisfying the conditions in the statement. We prove that $L_K(E)$ is decomposable by showing that conditions (a) and (b) in Theorem \ref{thm:decomp} are satisfied for $H_1$.

We start by proving that condition (a) in Theorem \ref{thm:decomp} is satisfied. Let $x\in E^\infty$ be such that $x^0\subseteq E^0\setminus H_1$. By way of contradiction, assume that for every $n\in \N$ we have ${\rm Path}(r(e_n), H_1)\neq \emptyset$. We claim that $\vert x^0 \vert \neq \infty$. Assume that this is not the case. Since $s(x_n)$ connects to $H_1$, it cannot be in $H_2$ as $H_1\cap H_2= \emptyset$. Then $x^0\subseteq E^0\setminus (H_1\cup H_2)$. Since for every vertex in $E^0\setminus (H_1\cup H_2)$ there is at least one path which is $H_1$-compatible or $H_2$-compatible, if the number of vertices in $x^0$ is infinite, then there are infinitely many paths starting at $s(x)$ which are either $H_1$-compatible or $H_2$-compatible, a contradiction. 
Since $\vert x^0 \vert$ is finite, there is a cycle based at a vertex in $x^0$, contradicting again that  for any vertex in  $E^0\setminus (H_1\cup H_2)$ there are finitely many paths  which are either $H_1$-compatible or $H_2$-compatible.

Finally we prove that condition (b) in Theorem \ref{thm:decomp} is satisfied. Assume that there is a vertex $v$ such that $ |\{e\in E^1 \ \vert \ s(e) =v \text{ and } \operatorname{Path}(r(e),H_1) \neq \emptyset\}| = \infty$. Denote by $C:=\{e\in E^1 \ \vert \ s(e) =v \text{ and } \operatorname{Path}(r(e),H_1) \neq \emptyset\}$. We prove that $v \in H_1 \cup  B_{H_1}.$ By way of contradiction, assume $v\notin H_1 \cup  B_{H_1}$. 
For every $e\in C$, let $\mu_e$ be a path starting by $e$ such that $r(\mu_e)\in H_1$. We may assume that $\mu_e\in F_E(H_1)$.  By the hypothesis there is only a finite number of them which are $H_1$-compatible. Therefore, there are infinitely many which are not $H_1$-compatible. Fix one of these, say $\lambda$. Write $\lambda=\lambda_1\dots \lambda_{\vert \lambda\vert}$, where $\lambda_i\in E^1$. Since $\lambda\in F_E(H_1) $ we may say that $s(\lambda_{\vert \lambda\vert})\notin H_1$, therefore $s(\lambda_{\vert \lambda\vert})\in B_{H_1}$. Note that $s(\lambda_{\vert \lambda\vert})\notin H_2$ because otherwise $r(\lambda)\in H_1\cap H_2=\emptyset$. Apply that $s(\lambda_{\vert \lambda\vert})\notin H_1\cup H_2$ and the hypothesis to find an $H_2$-compatible path $\rho$. Then $\lambda\rho$ is $H_2$-compatible. Now, the set $\{\lambda\rho\}$ is an infinite set of $H_2$-compatible paths starting at the vertex $v$, which is a contradiction to our hypothesis. 
\end{proof}

%%%%%%%%%%%%%%%%%%%%%%%%%%%%%%

\end{document}